\documentclass[12pt]{article}

\usepackage{amssymb,amsmath,accents}
\usepackage{amsfonts,amsthm,mathrsfs}
\usepackage{mathtools}
\usepackage{empheq}
\usepackage[dvipdfmx]{graphicx}
\usepackage{subcaption}
\usepackage{diagbox}
\usepackage{url}
\usepackage{here} % 図の位置
\mathtoolsset{showonlyrefs}
\numberwithin{equation}{section}
\newtheorem{thm}{Theorem}[section]
\newtheorem{cor}[thm]{Corollary}
\newtheorem{lemma}[thm]{Lemma}
\newtheorem{prop}[thm]{Proposition}

\newtheorem{remark}[thm]{Remark}

\DeclareMathOperator{\sn}{sn}

\begin{document}
\title{\Large\bf On periodic traveling wave solutions with or without phase transition to the Navier-Stokes-Korteweg and the Euler-Korteweg equations}

\author{
Yoshikazu Giga$^1$\footnote{Corresponding author. E-mail: labgiga@ms.u-tokyo.ac.jp}, %\\
%{\small  $^1$ Graduate School of Mathematical Sciences, The University of Tokyo}\\
%{\small  3-8-1 Komaba, Tokyo 153-8914, Japan;}\\
 Takahito Kashiwabara$^2$\footnote{E-mail: tkashiwa@ms.u-tokyo.ac.jp}  %\\
%{\small $^2$ Mathematical Sciences, YYY University}\\
%{\small XXXX xxx-xxxx, Japan} \\
 and Haruki Takemura$^3$\footnote{E-mail: takemura@ms.u-tokyo.ac.jp} %\\
%{\small $^3$ Mathematical Sciences, ZZZ University} \\
%{\small XXXX xxx-xxxx, Japan} \\
}

\date{}
%\today}

\maketitle
\noindent\textbf{Abstract.} The Navier-Stokes-Korteweg and the Euler-Korteweg equations are considered in isothermal setting. 
 These are diffuse interface models of two-phase flow.
 For the Navier-Stokes-Korteweg equations, we show that there is no periodic traveling wave solution with phase transition although there exists a non-constant periodic traveling wave solution with no phase transition.
 For the Euler-Korteweg equations, we show that there always exists a periodic traveling wave solution with phase transition for any period if the Korteweg relaxation parameter is small compared with the period provided that the available energy is double-well type.
 We also show that such a periodic traveling wave solution tends to a monotone traveling wave solution as the period tends to infinity under suitable spatial translation.
 Our numerical experiment suggests that there is periodic traveling wave with phase transition which is stable under periodic perturbation for small viscosity but it seems that this is a transition pattern.

%%%%%%%%%%%%%%
% 原稿 2024/10/7、1/8
\section{Introduction} \label{SIn} % Section 1

The Navier-Stokes-Korteweg equations are a typical diffuse interface model to describe two phase flow consisting vapor and liquid.
  It was proposed by Korteweg \cite{Ko} based on the idea of van der Waals \cite{VW} who introduced an energy describing coexisting two phases.
 Dunn and Serrin \cite{DS} derived a modern form of the equations from fundamental principles of thermomechanics by introducing interstitial work.
 Heida and M\'alek \cite{HM} derived such models by entropy production method.
 For further derivations of models the reader is refered to \cite{FK}, \cite{MP} and references therein.

For diffuse interface model, it is important to check what kind of sharp interface model is approximated.
 We consider its isothermal model.
 We calculate its evolution for one-dimensional model with periodic boundary condition numerically.
 According to our numerical experiments, it looks there exists rather stable periodic traveling wave with phase transition in the sense that the velocity is not spatially constant.
 However, as we show later, there exists no such a solution for the Navier-Stokes-Korteweg equations.
 We observe that our numerical solution provides a transit pattern and it is close a periodic traveling wave solution to the Euler-Korteweg equations, i.e., the equations with no viscosity.

The purpose of this paper has three folds.
 We first give a simple and rigorous proof that the Navier-Stokes-Korteweg equations admit no periodic traveling wave solution with phase transition.
 This is done just by integration by parts.
 We next show that for a given period there exists a periodic traveling
% 原稿 2024/10/7、2/8
wave solution with phase transition for the Euler-Korteweg equations provided that the Korteweg relaxation parameter and momentum are small.
 For the Navier-Stokes-Korteweg equations there always exists a periodic traveling wave solution with no phase transition but non-uniform density for small Korteweg relaxation parameter.
 We prove this fact by recalling a singular limit of the van der Waals' type energy with a double-well potential.
 We also prove that under suitable spatial translation our periodic traveling wave solution tends to a monotone traveling wave solution as the period tends to infinity.
 This is proved by a simple phase plane analysis.
 Finally, we compare with our numerical result.
 
We performed numerical computations for the initial value problem of the Navier-Stokes-Korteweg equations under periodic boundary conditions. Equilibrium states resembling traveling wave solutions with phase transitions were observed in the results. Furthermore, we confirmed that their profiles closely match the exact traveling wave solutions of the Euler-Korteweg equations. For the computations, we employed  a straightforward combination of the cubic interpolated pseudo-particle (CIP) method \cite{KaTa} with an explicit finite difference scheme. 

It is not difficult to find a monotone traveling wave solution with no periodic structure for the Euler-Korteweg equations.
 It corresponds to a heteroclinic orbit of the corresponding ordinary differential equations (ODE).
 Its existence can be proved by a phase plane analysis based on Hamiltonian structure.
 It is in a framework of general theory in \cite[pp.\ 11--12]{B} as pointed out by \cite{BDDJ}.
 It may be possible to apply this type of approach to find a periodic traveling wave but we rather apply a well-developed variational approach since we are especially interested in the case where the Korteweg relaxation parameter is small.
 The theory of singular limit of double-well potential with Dirichlet energy type relaxation is by now well developed.
 It was started by L.\ Modica \cite{MM1}, \cite{MM2} for one-dimensional setting.
 Later, it was extended to multi-dimensional setting by \cite{St}, \cite{KS}, \cite{Mo}.

The sharp interface limit is an interesting topic.
 It depends on scaling.
 For example, in \cite{ADKK}, under a conditional assumption it is proved that the solution of the (isothermal) Navier-Stokes-Korteweg equations converges to a weak solution of two-phase flow of incompressible Navier-Stokes equations with no phase transition provided that the square of the Mach number and the Korteweg relaxation parameter tends to zero with the same speed.
 The conditional assumption in \cite{ADKK} is that the associated energy functional converges to the interface energy as the above parameter tends to zero.
In \cite{DGKR}, a sharp interface limit result with a different scaling, where the Mach number remains non-zero, is presented based on the method of matched asymptotic expansions.
 Long time ago, to describe (constant-speed) motion of phase boundaries in two phase fluids, Slemrod \cite{Sl} proposed to use traveling wave solutions of a version of the Navier-Stokes-Korteweg equations.
 Later, Freist\"uhler \cite{F} explained motion of phase boundaries (in mixed fluids) with or without phase transition by using traveling wave solutions to the Navier-Stokes-Cahn-Hilliard equations.

For initial-value problem for the (isothermal) Navier-Stokes-Korteweg equations by now there are many well-posedness results both for strong and weak solutions.
 It was started by Hattori and Li \cite{HL1}, \cite{HL2}, who proved the unique existence of a global solutions with small initial data by assuming the pressure is strictly monotone which avoid phase transition.
 A similar result has been proved in \cite{DD} in the Besov space setting.
% 原稿 2024/11/14, 8, 1/5
 Also, one is able to construct a solution by maximum regularity theory \cite{Kot}, \cite{MSh}.
 These results are for construction of a strong solution, which is a smooth solution if the data is smooth.
 There are also various results to construct a weak solution for a general data.
 See, for example, \cite{BDL}, \cite{Ha}.
 However, research on the case of non-monotone pressure is so far quite limited both for strong solutions and weak solutions.
 A strong solution is constructed by \cite{KT} near the density of zero sound speed, that is, a critical point of the pressure.
 It also discusses its stability.
 Stability of constant states has been discussed in various papers for monotone pressure see e.g.\ \cite{TZ}.
 At this moment, we do not know the stability of periodically traveling wave solution.

% 原稿 2024/11/14, 8, 2/5
For the Euler-Korteweg equations, global well-posedness has been established for one-dimensional setting including non-monotone pressure \cite{BDD1}.
 The local well-posedness is established and blow-up criteria are given in \cite{BDD2}.
 These results are for strong solutions.
 A weak global-in-time solution is given by \cite{AM1}, \cite{AM2} but for the Korteweg parameter $\varepsilon=\varepsilon(\rho)$ is taken as $\mathrm{const}/\rho$, the case of quantum fluids.
 The stability of constant states with zero sound speed has been established by \cite{Son} in multi-dimensional setting.
 In one-dimensional setting of homoclinic orbits called soliton has been established in \cite{BDDJ} based on Hamiltonian structure.
 The stability of heteroclinic orbits called kink is also discussed in \cite{BDDJ} under some conditional assumptions.
 It is not clear that their result apply to our periodically traveling waves.

This paper is organized as follows.
 In Section \ref{SNP}, we prove non-existence of a periodic traveling wave solution with phase transition for the Navier-Stokes-Korteweg equations.
 In Section \ref{SE}, we prove the existence of a periodic traveling wave solution for the Euler-Korteweg equations based on variatonal method.
 In Section \ref{SMT}, we prove that a periodic traveling wave solution tends to a monotone traveling wave solution as the period tends to infinity.
 In Section \ref{SNE}, we compare with numerical results.
 In Section \ref{SApp}, we give our numerical scheme.

%%%%%%%%%%%%%%
\section{No phase transition for periodic traveling wave solutions for viscous case} \label{SNP} % Section 2

The isothermal Navier-Stokes-Korteweg equations are of the form
\begin{empheq}[left={\empheqlbrace}]{alignat=2}
	&\partial_t \rho + \operatorname{div} m = 0 
	\quad\text{in}\quad \mathbb{R}^n \times (0,\infty), \label{EM} \\
	&\partial_t m + \operatorname{div} \left(m\otimes m/\rho\right) - \operatorname{div} S + \nabla p = \rho\nabla\varepsilon\Delta\rho
	\quad\text{in}\quad \mathbb{R}^n \times (0,\infty). \label{EMME}
\end{empheq}
Here $\rho$ denotes the density and $m$ denotes the momentum so that $u=m/\rho$ is the velocity.
 The parameter $\varepsilon=\varepsilon(\rho)$ denotes the Korteweg relaxation parameter.
 It is assumed to be positive.
 The symbol $S$ denotes the viscous stress tensor while $p=p(\rho)$ denotes the pressure.
 If $S\equiv0$, we call these equations the Euler-Korteweg equations.
 We consider the Newtonian fluid, i.e., we assume that $S$ is of the form
\[
	S = \mu (\nabla u + \nabla u^T)
	+ \lambda (\operatorname{div} u) I,
\]
where $I$ denotes the identity.
 The shear viscosity $\mu=\mu(\rho)$ and the bulk viscosity $\lambda=\lambda(\rho)$ is assumed to satisfy
\[
	\bar{\mu} = n\lambda + 2\mu > 0
\]
so that the problem for $u$ 
% 原稿 2024/10/7、3/8
is parabolic.
 Physically speaking, the condition $\bar{\mu}\ge0$ is necessary so that the system fulfills the second law of thermodynamics.
 Although we assume that the pressure $p=p(\rho)$ is locally Lipschitz for $\rho>0$, we do not assume the monotonicity of $p$ with respect to $\rho$ so that it allows two phases and phase transition.
 We also assume that $\varepsilon(\rho)$ is continuous for $\rho\ge0$.

We say that $(\rho,m)$ is periodic if there is a period $T>0$ such that
\[
	\left( \rho(x,t), m(x,t) \right)
	= \left( \rho(x,t+T), m(x,t+T) \right). 
\]
We say that a solution $(\rho,m)$ to \eqref{EM}--\eqref{EMME} is a traveling wave solution with speed $c$ if there is $c\in\mathbb{R}^n$ and profile functions $\bar{\rho}$, $\bar{m}$ such that 
\[
	\left( \rho(x,t), m(x,t) \right)
	= \left( \bar{\rho}(x-ct), \bar{m}(x-ct) \right). 
\]
If we look at the problem in a moving frame with speed $c$, a periodic traveling wave solution with period $T$ and speed $c$ is nothing but a stationary solution with periodic boundary condition with period $cT$.
 This procedure is done by considering the Galilean transformation $\tilde{x}=x-ct$, $\tilde{u}=u-c$.

We consider the problem in one-dimensional setting.
 The equation \eqref{EM}--\eqref{EMME} is reduced to
\begin{empheq}[left={\empheqlbrace}]{alignat=2}
	&\partial_t \rho + \partial_x m = 0 
	&\quad\text{in}\quad \mathbb{R} \times (0,\infty) \label{EM1} \\
	&\partial_t m + \partial_x (m^2/\rho) + \partial_x p
	- \partial_x(\bar{\mu}\partial_x u) = \rho \partial_x(\varepsilon \partial_x^2 \rho)
	&\quad\text{in}\quad \mathbb{R} \times (0,\infty) \label{EMME1}
\end{empheq}
with $\bar{\mu}=\lambda+2\mu>0$.
 Here $m$ and $u$ are now scalar functions.
 A stationary solution to \eqref{EM1}--\eqref{EMME1} must solve
\begin{empheq}[left={\empheqlbrace}]{alignat=2}
	&\partial_x m = 0 
	&\quad\text{in}\quad \mathbb{R}, \label{EMS} \\
	&\partial_x (m^2/\rho) + \partial_x p
	- \partial_x(\bar{\mu}\partial_x u) = \rho \partial_x(\varepsilon \partial_x^2 \rho)
	&\quad\text{in}\quad \mathbb{R}. \label{EMMS}
\end{empheq}
The first equation \eqref{EMS} says that $m$ is a constant.
% 原稿 2024/10/7、4/8
\begin{thm} \label{TNon}
Assume that $\bar{\mu}(\rho)>0$ and $\varepsilon(\rho)>0$ for $\rho>0$.
 Let $(\rho,m)$ be a $C^2$ solution to \eqref{EMS}--\eqref{EMMS} with periodic boundary condition and $\rho\ge0$.
 Then $\rho$ must be a constant provided that $m\neq0$.
\end{thm}
\begin{proof}
It is convenient to introduce the mass specific volume $v=1/\rho$ and the mass specific available energy $\psi(\rho)$ which is defined by $\psi(\rho)=\tilde{\psi}(1/\rho)$ with $\partial_v\tilde{\psi}(v)=-\tilde{p}(v)$, where $\tilde{p}(v)=p(1/v)$.
 (Note that $\psi$ is determined up to constant from $p$.) 
 If we introduce the volume specific available energy $\Psi$ by $\Psi(\rho)=\rho\psi(\rho)$, then
\begin{equation} \label{EPr}
	p(\rho) = \rho^2 \partial_\rho \psi(\rho) = \rho \Psi'(\rho) - \Psi(\rho),
\end{equation}
where $\Psi'(\rho)=\partial_\rho\Psi(\rho)$.
 A direct calculation shows that $\partial_\rho p(\rho)=\Psi''(\rho)\rho$ so that
\[
	\nabla p \left(\rho(x) \right)
	= (\partial_\rho p) \nabla \rho
	= \rho \nabla \Psi'(\rho).
\]

The equation \eqref{EMMS} becomes
\[
	\partial_x (m^2 v) + \frac1v \partial_x \Psi'(\rho)
	- \partial_x \left( \bar{\mu} \partial_x (mv) \right)
	= \frac1v \partial_x (\varepsilon \partial_x^2 \rho).
\]
Since $m$ is a constant, multiplying both sides with $v$ yields
\begin{equation} \label{EKS}
	\partial_x \left( \frac{m^2 v^2}{2} \right)
	+ \partial_x \Psi'(\rho)
	- v\partial_x \left( \bar{\mu} \partial_x(mv) \right)
	= \partial_x (\varepsilon \partial_x^2 \rho).
\end{equation}
Here we invoke $v\partial_x v=(\partial_x v^2)/2$.
 We next notice that
% 原稿 2024/10/7、5/8
\[
	v\partial_x (\bar{\mu} \partial_x v)
	= \partial_x (v \bar{\mu} \partial_x v) - \bar{\mu}\partial_x v \partial_x v.
%	\partial_x (\bar{\mu} v) &= \bar{\mu} \partial_x v	+ \left(\partial_x \left( \bar{\mu} \left(\rho(x) \right)\right)\right) v \\
%	&= \bar{\mu} \partial_x v - (\partial_\rho \bar{\mu} ) \rho \partial_x v
\]
Plugging this identity to \eqref{EKS}, we obtain
\begin{equation} \label{ESE}
	\partial_x \left\{ \frac{m^2 v^2}{2} + \Psi'(\rho) 
	- v \bar{\mu} \partial_x v - \varepsilon \partial_x^2 \rho \right\}
	+ m \bar{\mu} |\partial_x v|^2 = 0.
\end{equation}
Let $\omega>0$ be a period.
 Integrating \eqref{ESE} over $(0,\omega)$ and using the periodicity, we end up with
\[
	m \int_0^\omega \bar{\mu} |\partial_x v|^2 \;dx = 0.
\]
By our assumption on viscosity, this implies $\partial_xv=0$, i.e., $\rho$ is constant provided that $m\neq0$.
\end{proof}

Let us go back to a periodic traveling wave.
 It is convenient to introduce mass flux
\[
	j = \rho u - \rho c = m - \rho c
\]
in one-dimensional setting.
 The same quantity is defined on the interface in sharp interface model; multi-dimensional setting we have to take its normal component with respect to the interface \cite{PS}.
% 原稿 2024/10/7、6/8
 We say that there exists no phase transition for a traveling wave solution if $j\equiv0$.
 Otherwise, we say that there exists phase transition.
 In other words, $u=c$ if and only if there is no phase transition.
 The velocity of the fluids must agree with velocity of the traveling wave,
 In the moving frame, this is equivalent to saying that $m=0$.
 Thus, Theorem \ref{TNon} easily deduces
\begin{cor} \label{CNon}
Assume that $\bar{\mu}(\rho)>0$ and $\varepsilon(\rho)\ge0$ for $\rho>0$.
 Then there exists no phase transition for periodic traveling wave solution to \eqref{EM1}--\eqref{EMME1}.
\end{cor}

%%%%%%%%%%%%%%
\section{Existence of non-trivial periodic traveling wave solutions} \label{SE} % Section 3

We are interested in the case when the pressure $p$ is not necessarily increasing in $\rho$.
 If we introduce the mass specific available energy $\Psi$, then $p$ is of the form \eqref{EPr}.
 We are particularly interested in the case when $\Psi$ has both convex and concave parts.
 More precisely, we postulate that $\Psi=\Psi(\rho)$ is $C^2$ for $\rho>0$ and there is two points $\rho_\ell>\rho_g>0$ such that the graph of $\Psi$ has a bitangent line through $\bigl(\rho_g,\Psi(\rho_g)\bigr)$ and $\bigl(\rho_\ell,\Psi(\rho_\ell)\bigr)$ and it is situated above this bitangent line at least near $[\rho_g,\rho_\ell]$ except at $\rho_g$ and $\rho_\ell$; see Figure \ref{FBt}.
\begin{figure}[htp]
\centering
\includegraphics[width=6cm]{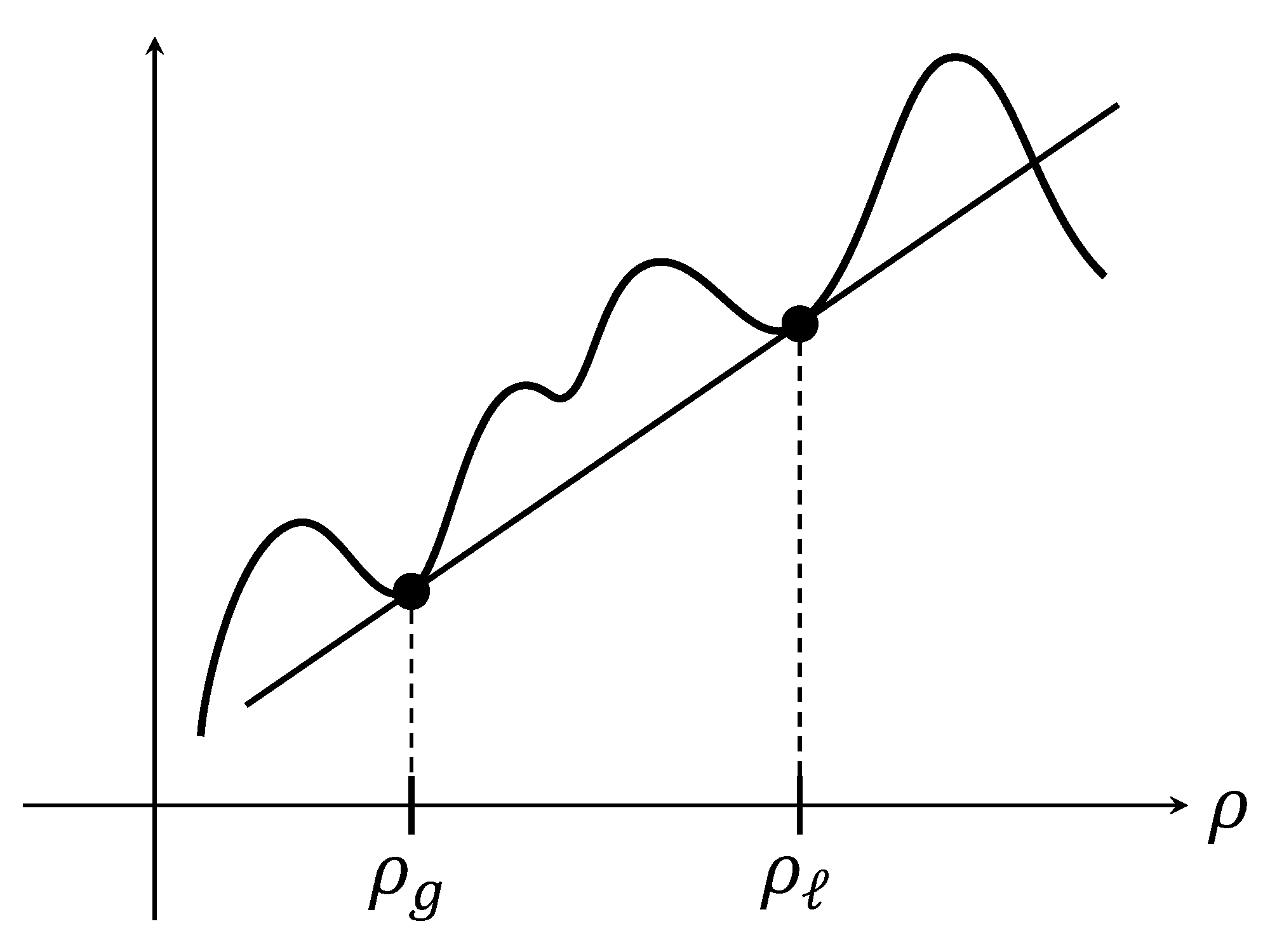}
\caption{graph of $\Psi$} \label{FBt}
\end{figure}

To include the case when the constant momentum $m\neq0$ for the Euler-Korteweg equations, it is 
% 原稿 2024/10/7、7/8
convenient to introduce a modified energy
\[
	\Psi^m(\rho) = \Psi(\rho) - \frac{m^2}{2\rho}
\]
as in \cite{GG}.
 By our assumption for $\Psi$, if $|m|$ is sufficiently small then $\Psi^m$ has the same property by replacing $\{\rho_\ell,\rho_g\}$ by different values $\{\rho_\ell^m,\rho_g^m\}$ close to $\{\rho_\ell,\rho_g\}=\{\rho_\ell^0,\rho_g^0\}$.
 In the sequel, we postulate that $m$ is taken so that this property holds.

Our goal in this section is to prove the existence of non-trivial periodic solution to \eqref{EMS}--\eqref{EMMS} including the case $m\neq0$.
\begin{thm} \label{TEx}
Assume that $\varepsilon(\rho)=\varepsilon$ is a positive constant and $\bar{\mu}=0$.
 Let $\rho_*$ be in the interval $(\rho_g^m,\rho_\ell^m)$.
 Then, there exists a (small) constant $\varepsilon_0>0$ such that for any $\omega>0$ satisfying $\varepsilon/\omega^2<\varepsilon_0$ there exists a non-constant periodic solution $\rho_\varepsilon$ to \eqref{EMS}--\eqref{EMMS} with period $\omega$ satisfying
\begin{equation} \label{EAv}
	\frac1\omega \int_0^\omega \rho\; dx = \rho_*.
\end{equation}
Moreover, the value of $\rho_\varepsilon$ is in $[\rho_\ell^m,\rho_g^m]$.
 Furthermore, one can take $\rho_\varepsilon$ so that it converges (as $\varepsilon\downarrow0$) to a function having exactly two jumps with values in $\{\rho_g^m,\rho_\ell^m\}$.
 In the case $m\neq0$, the velocity $u_\varepsilon=m/\rho_\varepsilon$ is not a constant function. 
 Even if $\bar{\mu}>0$, such a solution $\rho_\varepsilon$ exists if $m=0$ with constant $u_\varepsilon$.
\end{thm}
If we use $\Psi^m$, the equation \eqref{EMMS} (or \eqref{ESE}) can be written as
\begin{equation} \label{EG}
	\partial_x (\partial_\rho \Psi^m - \varepsilon \partial_x^2 \rho) = 0
	\quad\text{in}\quad \mathbb{T} = \mathbb{R}/(\omega\mathbb{Z}),
\end{equation}
when $\bar{\mu}=0$.
 In the case $\bar{\mu}$ is positive, we see $m=0$ by Theorem \ref{TNon} for the existence of non-trivial solution.

To prove Theorem \ref{TEx}, it suffices to find a non-constant solution of \eqref{EG} satisfying \eqref{EAv}.
 Since the problem is one dimensional, it may be possible to find such a solution by phase plane analysis for ODEs.
 However, in this paper we rather take a variational approach especially recalling the sharp interface limit of what is called the Allen-Cahn or the van der Waals' energy.
 Although this approach is by now standard, we give here at least outline of the proof for the reader's convenience and completeness.

% 原稿 2024/10/9、9 1/3
We begin with recalling a well-known $\Gamma$-convergence result originally due to Modica and Mortola \cite{MM1}, \cite{MM2}.
 We consider a double-well potential $W$ satisfying
\begin{enumerate}
\item[(W1)] $W\in C(\mathbb{R})$ and $W\ge0$,
\item[(W2)] $W(\tau)=0 \Leftrightarrow \tau=\pm1$,
\item[(W3)] $\liminf_{|\tau|\to\infty} W(\tau)>0$.
\end{enumerate}
We next consider a relaxed energy
\[
	E_\varepsilon(w) = \int_0^\omega \left\{ \frac12 \varepsilon \left| \frac{dw}{dx} \right|^2 + W(w) \right\} dx
\]
for $\varepsilon>0$.
 We consider this energy on a space of periodic functions with period $\omega>0$ under an average constraint.
 For $a\in(-1,1)$, let
\[
	X_a = \left\{ w \in L^1(\mathbb{T}) \biggm|
	\frac1\omega \int_0^\omega w\;dx = a \right\}, \quad
	\mathbb{T} = \mathbb{R}/(\omega\mathbb{Z}).
\]
We now consider a rescaled functional for $w\in X_a$
\begin{equation*}
	\mathcal{E}_\varepsilon(w) =
	\left\{
	\begin{array}{cl}
	E_\varepsilon(w)/\varepsilon^{1/2} &\text{if}\ E_\varepsilon(w) < \infty, \\
	\infty &\text{otherwise}.
	\end{array}
	\right.
\end{equation*}
The singular limit as $\varepsilon\to0$ turns to be
\begin{equation*}
	\mathcal{E}_0(w) =
	\left\{
	\begin{array}{cl}
	\sigma \ \#\ \text{jump of}\ w &\text{if}\ w\ \text{is a two-valued function with values}\ \pm 1, \\
	\infty &\text{otherwise}
	\end{array}
	\right.
\end{equation*}
for $w\in X_a$, where
\[
	\sigma = \int_{-1}^1 \sqrt{2W(\tau)}\; d\tau.
\]
% 原稿 2024/10/9、10 2/3
It is stated in a rigorous way as follows.
\begin{prop} \label{PMM}
Assume that $W$ satisfies (W1)--(W3).
 The functional $\mathcal{E}_0$ is the $L^1$-Gamma limit of $\mathcal{E}_\varepsilon$ on $X_a$ as $\varepsilon\downarrow0$.
 In other words,
\begin{enumerate}
\item[(i)] (lower semicontinuity) $\mathcal{E}_0(w)\le\liminf_{\varepsilon\downarrow0}\mathcal{E}_\varepsilon(w_\varepsilon)$ if $w_\varepsilon\to w$ in $L^1(\mathbb{T})$;
\item[(i\hspace{-0.1em}i)] (recovery sequence) for any $w\in X_a$ there exists a sequence $v_\varepsilon$ such that $v_\varepsilon\to w$ in $X_a$ and $\mathcal{E}_\varepsilon(v_\varepsilon)\to\mathcal{E}_0(w)$ as $\varepsilon\downarrow0$.
\end{enumerate}
\end{prop}

This result is usually stated without periodicity.
 However, its modification to the periodic setting is easy.
 We give a proof for (i) for the reader's convenience.
 For the lower semicontinuity, a key observation is
\begin{equation} \label{EMM}
	\mathcal{E}_\varepsilon(w) 
	\ge \int_0^\omega \left| \frac{dw}{dx} \right| \sqrt{2W(w)}\; dx
	= \int_0^\omega \left| \frac{dG(w)}{dx} \right| dx
\end{equation}
with $G(\tau)=\int_{-1}^\tau\sqrt{2W(q)}\; dq$, which follows from the Cauchy-Schwarz inequality $\alpha^2+\beta^2\ge2\alpha\beta$ with $\alpha=\varepsilon^{1/2}|dw/dx|$, $\beta=\sqrt{2W}\bigm/\varepsilon^{1/2}$.
 The above inequality \eqref{EMM} sometimes called the Modica-Mortola inequality.
 To show (i), we may assume that $\mathcal{E}_\varepsilon(w_\varepsilon)$ is bounded for $\varepsilon\in(0,1)$. This implies that the total variation of $r_\varepsilon=G(w_\varepsilon)$ is bounded by \eqref{EMM}, and also $\int_0^\omega W(w_\varepsilon)\;dx\to0$ as $\varepsilon\downarrow0$.
 By (W3) together with (W1) and (W2), we observe that $|w_\varepsilon|\to1$ in measure.
 By taking a subsequence, we see that $w_{\varepsilon'}\to\pm1$ a.e.\ so that $r_\varepsilon\to0$ or $\sigma$ a.e.
 By lower semicontinuity of total variation, we obtain (i).

% 原稿 2024/10/9、11 3/3
 To prove (i\hspace{-0.1em}i) we mollify $w$ so that jump part approximates a function satisfying $\varepsilon dw/dx=\sqrt{2W(w)}$ in a diffuse interface of length $O(\varepsilon^{1/2})$ as $\varepsilon\downarrow0$ up to a smaller interval.
 One has to be careful to keep the average condition and the periodicity.
 We do not give a detailed proof.
\begin{lemma} \label{LMM}
Assume that $W\in C^2(\mathbb{R})$ satisfies (W1) and (W2).
 Then there exists $\varepsilon_0>0$ (independent of $\omega$) such that if $\varepsilon/\omega^2<\varepsilon_0$, there is a solution $w=w_\varepsilon\in X_a$ of 
\begin{equation} \label{EEL}
	\frac{d}{dx} \left(\varepsilon \frac{d^2w}{dx^2} - W'(w) \right) = 0
\end{equation}
satisfying $|w_\varepsilon|\le1$ on $\mathbb{T}$ and $w_\varepsilon\to w_0$ in $L^1(\mathbb{T})$, where $w_0$ has exactly two jumps with values in $\{\pm1\}$.
\end{lemma}
% 原稿 2024/10/10、12 1/3
\begin{proof}[Sketch of the proof]
By scaling $x'=x/\omega$, $\varepsilon'=\varepsilon/\omega^2$, we may assume $\omega=1$.
 Since we are interested in a minimizer $w_\varepsilon$ satisfying $|w_\varepsilon|\le1$, we may modify $W(w)$ for $|w|>1$ such that $w$ is coercive in the sense
\[
	W(w) \ge c_0|w|^2 - c_1, \quad
	w \in \mathbb{R}
\]
with some positive constants $c_0$ and $c_1$.
 In particular, $W$ satisfies (W3).
\begin{enumerate}
\item[Step 1.] (Compactness). 
Let $w_\varepsilon$ be a minimizer of $\mathcal{E}_\varepsilon$ in $X_a$.
 The existence of a minimizer is guaranteed in $H^1(\mathbb{T})\cap X_a$ by coercivity.
 Since $\min\mathcal{E}_\varepsilon\to\min\mathcal{E}_0$ by Proposition \ref{PMM}, $\{w_\varepsilon\}$ is bounded in $L^2(\mathbb{T})$ and $\int_0^\omega W(w_\varepsilon)\;dx\to0$.
 Since the total variation of $r_\varepsilon=G(w_\varepsilon)$ is bounded, as in \cite[Proof of Theorem 2.2]{GOU}, we see, by coercivity, that $r_\varepsilon$ is bounded in $BV(\mathbb{T})$.
 By standard compactness in $BV(\mathbb{T})$ (see e.g.\ E.\ Giusti \cite{Giu}) we conclude that $r_\varepsilon\to r_*\in BV(\mathbb{T})$ in $L^1(\mathbb{T})$ and a.e.\ on $\mathbb{T}$ by taking a subsequence.
 Since $r_\varepsilon$ is bounded in $BV(\mathbb{T})$, $|r_\varepsilon|$ is uniformly bounded for $\varepsilon<1$ because of one-dimensionality.
 Since $G$ is one-to-one, this implies $w_\varepsilon\to G^{-1}(r_*)$ a.e.\ on $\mathbb{T}$ and $w_\varepsilon$ is uniformly bounded for $\varepsilon<1$.
 Since $\int_0^\omega W(w_\varepsilon)\;dx\to0$, we conclude that $w_\varepsilon\to\pm1$, a.e.\ on $\mathbb{T}$.
 We further observe that $w_\varepsilon\to w_*$ in $L^1(\mathbb{T})$ as $\varepsilon\downarrow0$ with some $w_*\in X_a$ and $\mathcal{E}_0(w_*)<\infty$.
\item[Step 2.] (Application of the Gamma convergence).
 Since $w_\varepsilon\to w_*$ in $L^1(\mathbb{T})$, we observe, by Proposition \ref{PMM}, that $w_*$ is a minimizer of $\mathcal{E}_0$ in $X_a$.
\item[Step 3.] (Structure of $w_*$).
 A minimizer $w_*$ of $\mathcal{E}_0$ must have at least one jump because $a\in(-1,1)$.
 By periodicity, $w_*$ has exactly two jumps because $w_*$ is a minimizer.
 Thus for sufficiently small $\varepsilon>0$, say $\varepsilon<\varepsilon_0$, $w_\varepsilon$ is not a constant.
\item[Step 4.] (Range of $w_\varepsilon$).
 For $w\in H^1(\mathbb{T})$, we consider a truncated function
\[
	\bar{w} = \max \left( \min(w,1), -1 \right).
\]
Then, as is well-known,
\[
	\left| \frac{d\bar{w}}{dx} \right|
	\le \left| \frac{dw}{dx} \right| \quad\text{a.e.}\quad
	x \in \mathbb{T}.
\]
 Since $W(\tau)>0$ for $|\tau|>1$, 
% 原稿 2024/10/10、13 2/3
we see that
\[
	\mathcal{E}_\varepsilon (\bar{w}) < \mathcal{E}_\varepsilon(w)
\]
if the set $\left\{ x\in\mathbb{T} \bigm| \left|w(x)\right|>1\right\}$ has a positive measure.
 Thus, our minimizer must satisfy
 \[
	|w_\varepsilon| \le 1 \quad\text{on}\quad \mathbb{T}.
\]
\item[Step 5.] (The Euler-Lagrange equations).
 Since we have a constant
\[
	\int_0^1 w\; dx = a,
\]
a minimizer of $\mathcal{E}_\varepsilon$ must satisfy the Euler-Lagrange equation with a Lagrange multiplier $\lambda$, which is a constant, of the form
\[
	-\varepsilon^{1/2} \frac{d^2w}{dx^2} + W'(w)/\varepsilon^{1/2} = \lambda
	\quad\text{in}\quad \mathbb{T}.
\]
Thus $w_\varepsilon$ must satisfy \eqref{EEL}.
\end{enumerate}
We now conclude that $w_\varepsilon$ satisfies all desired properties provided that $\varepsilon<\varepsilon_0$.
\end{proof}
\begin{proof}[Proof of Theorem \ref{TEx}]
We set
\[
	w = \frac{2}{\rho_g^m-\rho_\ell^m} (\rho - \rho_g^m) - 1
\]
and
\[
	W(w) = \Psi^m(\rho) - \ell_m(\rho)
	\quad\text{for}\quad |w| \le 1+\delta
\]
where
\[
	\ell_m(\rho) = (\partial_\rho \Psi^m) (\rho_g^m) (\rho-\rho_g^m) + \Psi^m(\rho_g^m)
\]
whose graph is a bitangent line of $\Psi^m$ going through $\bigl(\rho_g^m,\Psi^m(\rho_g^m)\bigr)$ and $\bigl(\rho_\ell^m,\Psi^m(\rho_\ell^m)\bigr)$.
% 原稿 2024/10/10、14 3/3
 Here $\delta$ is taken small so that $\Psi^m(\rho)>\ell_m(\rho)$ for $|w|\le1+\delta$.
 We extend $W$ smoothly so that $W(w)>0$ for $|w|>1+\delta$.
 Then, Lemma \ref{LMM} yields Theorem \ref{TEx}.
\end{proof}
Theorem \ref{TEx} yields existence of a non-trivial periodic traveling wave solution with phase transition.
\begin{cor} \label{CEx}
Assume the same hypotheses of Theorem \ref{TEx} concerning $\varepsilon$, $\rho_*$ and $\bar{\mu}$ so that $\bar{\mu}=0$.
 Then there exists a small constant $\varepsilon_0$ such that for any $\omega>0$ there exists an $\omega$-periodic traveling wave solution to \eqref{EM1}--\eqref{EMME1} satisfying \eqref{EAv} with phase transition provided that $\varepsilon/\omega^2<\varepsilon_0$.
\end{cor}
\begin{remark} \label{RE}
Of course, our theory always implies that the existence of non-trivial periodic traveling wave solutions with no phase transition for the Euler-Koreteweg equations ($\bar{\mu}=0$) and the Navier-Stokes-Korteweg equations ($\bar{\mu}>0$).
\end{remark}

%%%%%%%%%%%%%%
% 原稿 2024/12/25、1/5
\section{Monotone traveling wave} \label{SMT} % Section 4

We are interested in a traveling wave solution with or without phase transition which is monotone in $\mathbb{R}$.
 We shall prove that our periodic traveling wave solution tends to such a solution as the period tends to the infinity.
\begin{prop} \label{PMon}
Assume the same hypotheses of Theorem \ref{TEx} with $\bar{\mu}=0$.
 Then, for any $\varepsilon>0$, there exists a unique (up to translation) solution $\rho_\varepsilon^\infty$ \eqref{EMS}--\eqref{EMMS} in $\mathbb{R}$ satisfying
\[
	\lim_{x\to-\infty} \rho_\varepsilon^\infty(x) = \rho_g^m, \quad
	\lim_{x\to+\infty} \rho_\varepsilon^\infty(x) = \rho_\ell^m.
\]
The solution $\rho_\varepsilon^\infty$ is non-decreasing in $x$.
 In the case $m\neq0$, the velocity $u_\varepsilon^\infty=m/\rho_\varepsilon^\infty$ is not a constant function in $x$.
 (Even if $\bar{\mu}>0$, such a unique solution exists if $m=0$ with constant speed $u_\varepsilon^\infty$.)
\end{prop}
% 原稿 2024/12/25、2/5
\begin{proof}
This fact is already noticed in \cite{BDDJ} by noting that the system is Hamiltonian.
 Since the proof is elementary, we give it here for completeness.
 The equation \eqref{EMMS} can be written as \eqref{EG} on $\mathbb{R}$ instead of $\mathbb{T}$.
 As in the proof of Theorem \ref{TEx}, by changing dependent variables, the problem is reduced to solve
\begin{equation} \label{EL}
	\varepsilon \frac{d^2w}{dx^2} - W'(w) = \lambda \quad
	\text{(}\lambda\text{: constant) in} \quad \mathbb{R}
\end{equation}
with
\[
	w(-\infty) = -1, \quad
	w(+\infty) = +1
\]
where $w(\pm\infty)=\lim_{x\to\pm\infty}w(x)$.
 Multiplying $w_x=dw/dx$ with \eqref{EL}, we get
\begin{equation} \label{EHa}
	\frac{d}{dx} H(w,w_x) = 0
\end{equation}
with $H(q,p)=\varepsilon p^2/2-W(q)-\lambda q$.
 The problem is to find a heteroclinic orbit for \eqref{EHa} connecting $w(-\infty)=-1$ to $w(+\infty)=+1$.
 Since $W'$ is minimized at $\pm1$ with equal depth, to find such orbit $\lambda$ must be zero.
 Our assertion is obtained by a simple phase plane analysis.
 There is only one level set of $H$ connecting $(\pm1,0)$ and this corresponds the orbit of heteroclinic orbit.
 The monotonicity is clear since on the orbit $p=w_x$ is always positive if $w(-\infty)=-1$. 
\end{proof}
% 原稿 2024/12/25、3/5
By using a moving frame to Proposition \ref{PMon}, one finds a monotone traveling wave solution for the Euler-Korteweg equations with phase transition.
\begin{thm} \label{TMon}
Assume the same hypotheses of Theorem \ref{TEx} with $\bar{\mu}=0$.
 Assume that $u_1,u_2\in\mathbb{R}$ satisfies $u_2-u_1=m(1/\rho_\ell^m-1/\rho_g^m)$.
 Then, for any $\varepsilon>0$, there exists a unique (up to translation) traveling wave solution $(\rho_\varepsilon^\infty,u_\varepsilon^\infty)$ of \eqref{EM} and \eqref{EMME} with $\bar{\mu}=0$ such that
\[
	\rho_\varepsilon^\infty(-\infty) = \rho_g^m, \quad
	\rho_\varepsilon^\infty(+\infty) = \rho_\ell^m, \quad
	u_\varepsilon^\infty(-\infty) = u_1, \quad
	u_\varepsilon^\infty(+\infty) = u_2
\]
with speed $c=u_2-m/\rho_g^m=u_1-m/\rho_\ell^m$;
 if $m\neq0$, this solution admits phase transition.
 (In the case $\bar{\mu}>0$, there still exists a unique (up to transition) traveling wave solution with no phase transition, i.e., $c=u_1=u_2$ with $m=0$.)
\end{thm}
In the case of $\bar{\mu}>0$, there seems to be no traveling wave solutions with phase transition under suitable assumptions of the space infinity.

We shall prove that our periodic traveling wave solution tends to monotone traveling wave solution as the period tends to infinity.
 To simplify the argument, we further assume that
\[
	W(w) = \Psi^m(\rho) - \ell_m(\rho)
\]
in the proof of Theorem \ref{TEx} has only one critical point $w_*\in(-1,1)$ on $(-1,1)$ with $W''(w_*)<0$;
 see Figure \ref{FPh}.
\begin{figure}[H]
\centering
\includegraphics[width=8.2cm]{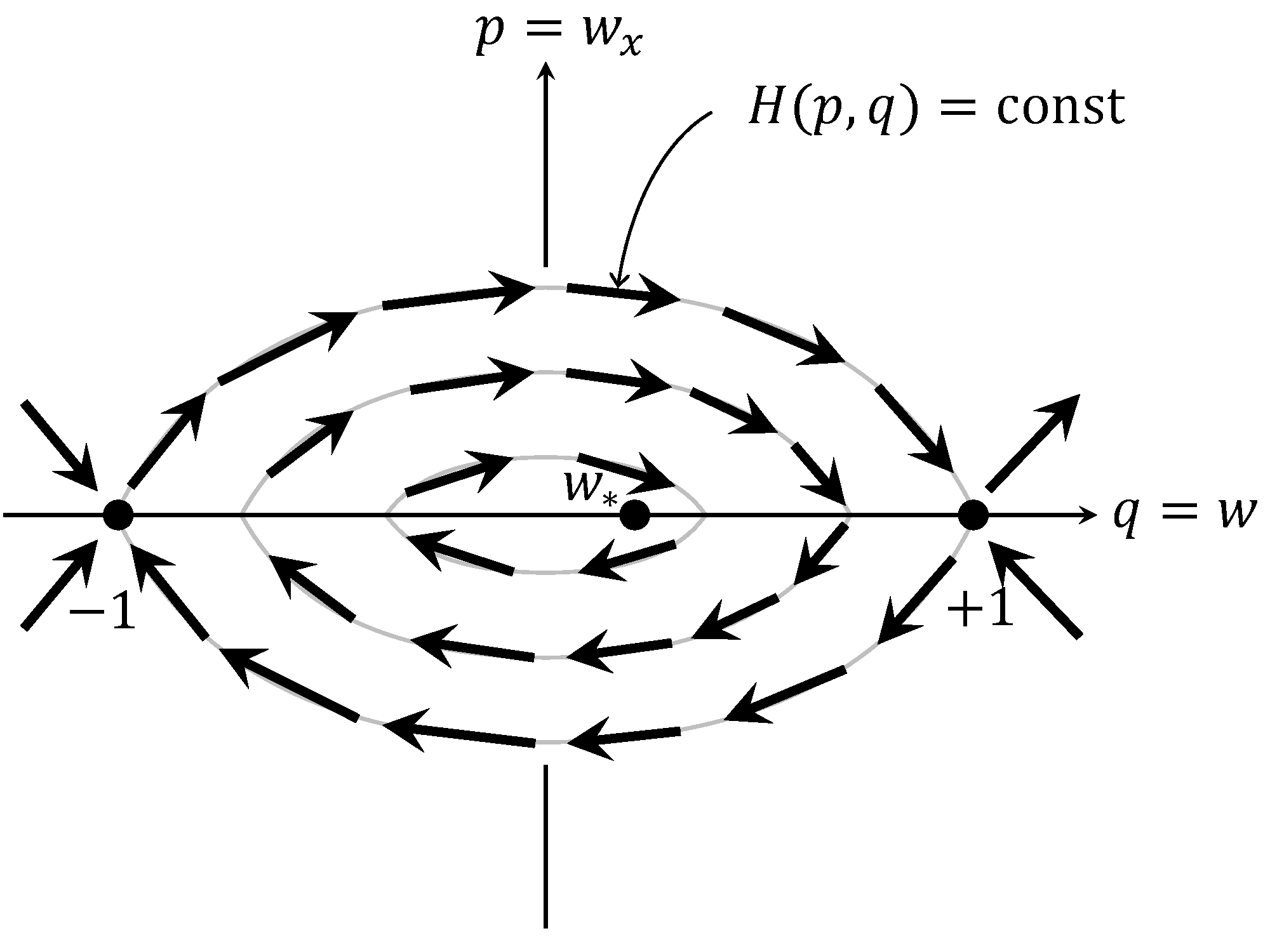}
\caption{phase plane for $\lambda=0$} \label{FPh}
\end{figure}
\begin{lemma} \label{LPL}
Assume $W$ has a unique critical point $w_*\in(-1,1)$ on $(-1,1)$ with $W''(w_*)<0$.
 Assume the same hypotheses of Lemma \ref{LMM}.
 Let $x_0$ be a jump point of $w_0$ so that $w_0(x_0+0)>w_0(x_0-0)$.
 Then the function
\[
	w_\varepsilon^\omega(x) := w_\varepsilon(x_0+x)
\]
% 原稿 2024/12/25、4/5
converges to a solution $w_\varepsilon^\infty$ of
\[
	\varepsilon \frac{d^2w}{dx^2} - W'(w) = 0
	\quad\text{on}\quad \mathbb{R}
\]
with $w(\pm\infty)=\pm1$ as $\omega\to\infty$ by taking a suitable subsequence. 
 If we set $w_\varepsilon^\omega(x)=w_\varepsilon(x_0+x+\delta_\omega)$ with some $\delta_\omega\in\mathbb{R}$, the convergence becomes a full convergence.
\end{lemma}
\begin{proof}
We may assume that $x_0=0$.
 Since
\[
	w_\varepsilon^\omega (x) = w_{\varepsilon/\omega^2}^1(x/\omega)
\]
and $w_\varepsilon^\omega$ satisfies
\begin{equation} \label{EPW}
	\varepsilon \frac{d^2w}{dx^2} - W'(w) = \lambda_{\varepsilon,\omega} \quad\text{on}\quad \mathbb{R}
	\quad\text{with}\quad w(x+\omega) = w(x),
\end{equation}
integrating this differential equation over $\mathbb{T}$ to get
\[
	\omega\lambda_{\varepsilon,\omega} 
	= -\int_0^\omega W' \left( w_\varepsilon^\omega(x) \right) dx.
\]
Since $w_{\varepsilon'}^1(y)\to1$ or $-1$ a.e.\ $y\in\mathbb{R}$ with $|w_{\varepsilon'}^1|\le1$ as $\varepsilon'\to0$, 
\[
	\lambda_{\varepsilon,\omega} 
	= -\frac1\omega \int_0^\omega W' \left( w'_{\varepsilon/\omega^2}(x/\omega) \right) dx
	= -\int_0^1 W' \left( w_{\varepsilon/\omega^2}^1 (y) \right) dy \to 0
\]
as $\omega\to\infty$ by the dominated convergence theorem.
 We fix $\varepsilon>0$ in \eqref{EPW} and let $\omega\to\infty$.
 The equation \eqref{EPW} gives a bound for $w_{xx}$ so that $w_x$ is bounded.
 Thus, by Arzel\`{a}-Ascoli theorem $w_\varepsilon^\omega\to w_\varepsilon^\infty$ for some $w_\varepsilon^\infty$ as $\omega\to\infty$ with its first derivative locally uniformly in $\mathbb{R}$ by taking a subsequence.
 Thus, the limit $w_\varepsilon^\infty$ must satisfy
% 原稿 2024/12/25、5/5
\begin{equation} \label{EWh}
	\varepsilon \frac{d^2w}{dx^2} - W' (w) = 0
	\quad\text{in}\quad \mathbb{R}
\end{equation}
with $|w|\le1$ since $\lambda_{\varepsilon,\omega}\to0$ as $\omega\to\infty$.

We shall prove that $w_\varepsilon^\infty$ is not a constant for fixed $\varepsilon>0$.
 If $w_\varepsilon^\infty$ were a constant function, it must be either $\pm1$ or $w_*$.
 We observe that up to subsequence
\begin{equation} \label{EACON}
	\lim_{\omega\to\infty} w_{\varepsilon/\omega^2}^1(y)
	= \operatorname{sgn} y \quad\text{for a.e.}\quad
	y \in (-\omega/2,\omega/2).
\end{equation}
We notice that a periodic solution is a closed orbit of the phase plane Figure \ref{FPh} and its perturbation with small $\lambda$.
 Thus if it is close to $w_*$, the value of $w_\varepsilon^\omega$ is restricted in a small neighborhood of $w_*$.
 However, this is impossible since \eqref{EACON} holds.
 Thus $w_\varepsilon^\infty \not\equiv w_*$.
 In the phase plane Figure \ref{FPh}, the derivative of a periodic solution $w$ at $0$ is away from zero i.e., $(w_\varepsilon^\omega)_x(0)\ge c>0$ with some positive constant $c$ provided that $\omega$ is sufficiently large.
 This is because both $(-1,0)$ and $(+1,0)$ is close to the periodic orbit of $w_\varepsilon^\omega$ if $\omega$ is sufficiently large by \eqref{EACON}.
 The same assertion holds for the perturbed system with $\lambda$.
 Thus, $w_\varepsilon^\infty(0)\ge c$ which yields $w_\varepsilon^\infty\not\equiv\pm1$.

We may take $\delta_\omega$ so that $w_\varepsilon^\omega(0)$ equals the average of $w_\varepsilon^\omega$ over $\mathbb{T}$ which is independent of $\omega$, then the convergence is full convergence since a non-decreasing  solution is unique if we determine the value at $0$.
\end{proof}
This lemma immediately yields
\begin{thm} \label{TCMon}
Assume the same hypotheses of Theorem \ref{TEx}.
 Assume $W=\Psi^m-\ell$ has only one critical point $w_*\in(-1,1)$ with $W''(w_*)<0$.
 Let 
\[
	\rho_\varepsilon^\omega(x) = \rho_{\varepsilon/\omega^2}^1 (x+\delta_\omega/\omega),
\]
where $\rho_\varepsilon^1$ is a $1$-periodic solution of Theorem \ref{TEx}.
 With a suitable choice of $\delta_\omega\in\mathbb{R}$, the function converges to $\rho_\varepsilon^\infty$ in Theorem \ref{TMon} as $\omega\to\infty$.
\end{thm}

\section{Numerical experiments} \label{SNE}
% 第5章 (数値計算結果) とAppendix (数値スキームの説明): 

In this section, we present numerical experiments for the initial value problem of the one-dimensional isothermal Navier-Stokes-Korteweg equations
\begin{empheq}[left={\empheqlbrace}]{alignat=2} 
  & \partial_t \rho + \partial_x (\rho u) = 0 \quad & & \text{in} \quad \mathbb{T} \times [0,T], \label{IVP1} \\
  & \partial_t u + u \partial_x u + \rho^{-1} \bar{\mu} \partial_x^2 u + \partial_\rho^2 \Psi \partial_x \rho - \varepsilon \partial_x^3 \rho = 0 \quad & & \text{in} \quad \mathbb{T} \times [0,T], \label{IVP2} \\
  & u(x,0) = u_0(x), \quad \rho(x,0) = \rho_0 (x)  \quad & & \text{in} \quad \mathbb{T}, \label{IVP3}
\end{empheq}
where the Helmholtz free energy (available energy) is given by $ \Psi(\rho) = \frac{1}{4} (\rho - 1)^2 (\rho - 2)^2 $, $ \bar{\mu} $ is a positive constant, and $ \varepsilon $ is a small positive parameter. 
First two equations are equivalent with \eqref{EM1} and \eqref{EMME1}, using the relation $ m = \rho u $. Hereafter, we assume that the period $ \omega $ introduced in Section \ref{SNP} is equal to $ 1 $. We conducted numerical experiments with the method detailed in Appendix with the spatial mesh size $ h = 1 / 300 $ and the temporal step size $ \Delta t = 1 / 120000 $. 
Over time, the numerical solutions evolve into a traveling wave-like state, as illustrated in Figures \ref{fig:mu1eps4}--\ref{fig:mu3eps5}. 
We define the position of one of the interfaces at time $ t $, denoted by $ \bar{x}_h(t) $, as the point satisfying $ \rho_h(\bar{x}_h(t), t) = 0 $ and $ \partial_x \rho_h(\bar{x}_h(t), t) > 0 $. The velocity of the interface is defined by $ c_h(t) = \frac{d}{dt} \bar{x}_h(t) $. Figures \ref{fig:mu1eps4}, \ref{fig:mu3eps4}, \ref{fig:mu1eps5} and \ref{fig:mu3eps5} show the density $ \rho_h(x, t) $, the velocity $ u_h(x, t) $, the interface velocity $ c_h(t) $ and the mass flux $ \rho_h(x, t) (u_h(x, t) - c_h(t)) $ at $ t = 25 $ for the parameters $ (\bar{\mu}, \varepsilon) = (10^{-1}, 10^{-4}) , \ (10^{-3}, 10^{-4}) , \ (10^{-1}, 10^{-5}) $ and $ (10^{-3}, 10^{-5}) $ respectively. 

In all four results, the density splits into two phases. We can interpret that the domain where $ \rho \approx 1 $ as the vapor phase and the domain where $ \rho \approx 2 $ as the liquid phase. For $ \bar{\mu} = 10^{-1} $ (Figures \ref{fig:mu1eps4} and \ref{fig:mu1eps5}), the velocity $ u $ is nearly equal to the interface velocity in both phases, indicating almost no mass transfer between the two phases. For $ \bar{\mu} = 10^{-3} $ (Figures \ref{fig:mu3eps4} and \ref{fig:mu3eps5}), two domains with different velocities appear, as well as the density. The mass flux $ \rho (u - c) $ is approximately constant and positive throughout the domain, which implies there exists phase flux and the total mass is nearly conserved through the exchange of mass between the phases. 

The result of $ \bar{\mu} = 10^{-1} $, with an almost constant velocity, is consistent with Theorem \ref{TNon}. It seems that the equilibrium states with a non-constant velocity correspond to the non-trivial periodic solution of the Euler-Korteweg equations mentioned in Theorem \ref{TEx}. 

In the following, we demonstrate that the numerical solutions obtained above approximate the periodic traveling wave solution to the Euler-Korteweg equations mentioned in Section \ref{SE}. Let $ (\rho^{\text{EK}}, u^\text{EK}) $ be an arbitrary $ 1 $-periodic solution to the Euler-Korteweg equations, $ \rho_{l}^\text{EK} = \max_{x \in \mathbb{T}} \rho^\text{EK}(x) $, and $ \rho_{g}^\text{EK} = \min_{x \in \mathbb{T}} \rho^\text{EK}(x) $. The traveling speed of this solution is denoted by $ c^\text{EK} $, and its mass flux is given by $ m^\text{EK} =  \rho^{\text{EK}}(u^\text{EK} - c^\text{EK}) $. According to the proof of Theorem \ref{SE}, the tangent lines of $ \Psi^m = \Psi - \frac{m^2}{2\rho} $ at $ \rho_h $ and $ \rho_l $ are identical. 

Let  $ \rho_{h,l} = \max_{x \in \mathbb{T}} \rho_h $, $ \rho_{h,g} = \min_{x \in \mathbb{T}} $, $ \max_{x \in \mathbb{T}} \rho_h = \rho(x_l) $, $ \min_{x \in \mathbb{T}} \rho_h = \rho(x_g) $, $ u_g = u(x_g) $ and $ u_l = u(x_l) $. 
Figure \ref{fig:Psi} shows $ \Psi^m $ and its tangent lines at $ \rho_{h,l}  $ and $ \rho_{h,g} $. 
The two tangent lines approximately match in all four figures, therefore we conclude that the densities and velocities of the liquid and gas phases implied by the numerical solution are close to those of traveling wave solutions of the Euler-Korteweg equations. 

\begin{figure}[tbhp] 
  \centering
  \includegraphics[width=0.5\textwidth]{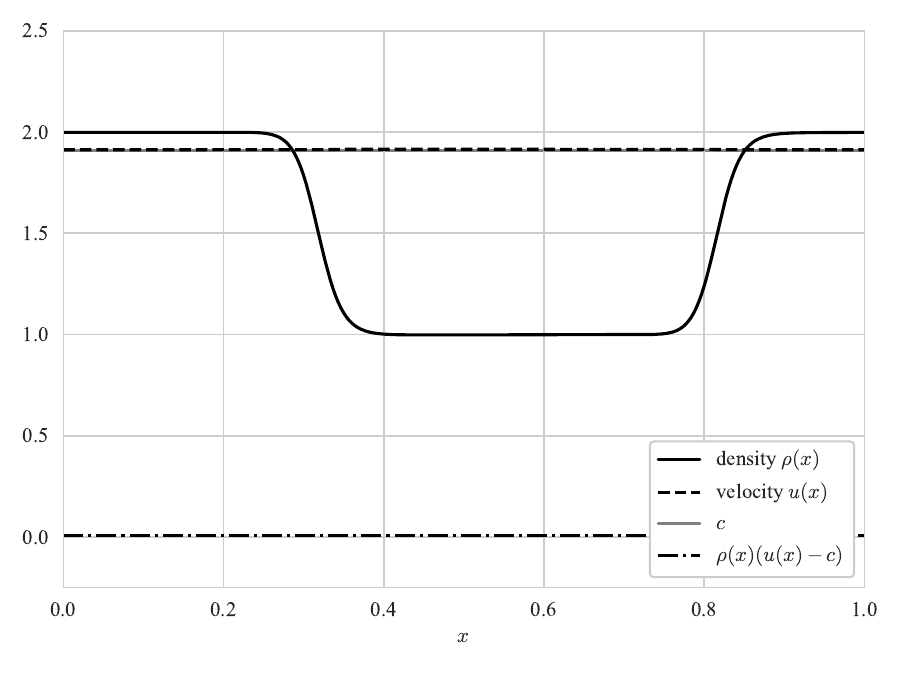}
  \caption{Graph of the density $ \rho(x, t) $, velocity $ u(x, t) $, interface velocity $ c(t) $, and mass flux $ \rho(x, t)(u(x, t) - c(t)) $ at time $ t = 25 $. The parameters are set to $ \bar{\mu} = 10^{-1} $ and $ \varepsilon = 10^{-4} $. }
  \label{fig:mu1eps4}
\end{figure}

\begin{figure}[tbhp]
  \centering
  \includegraphics[width=0.5\textwidth]{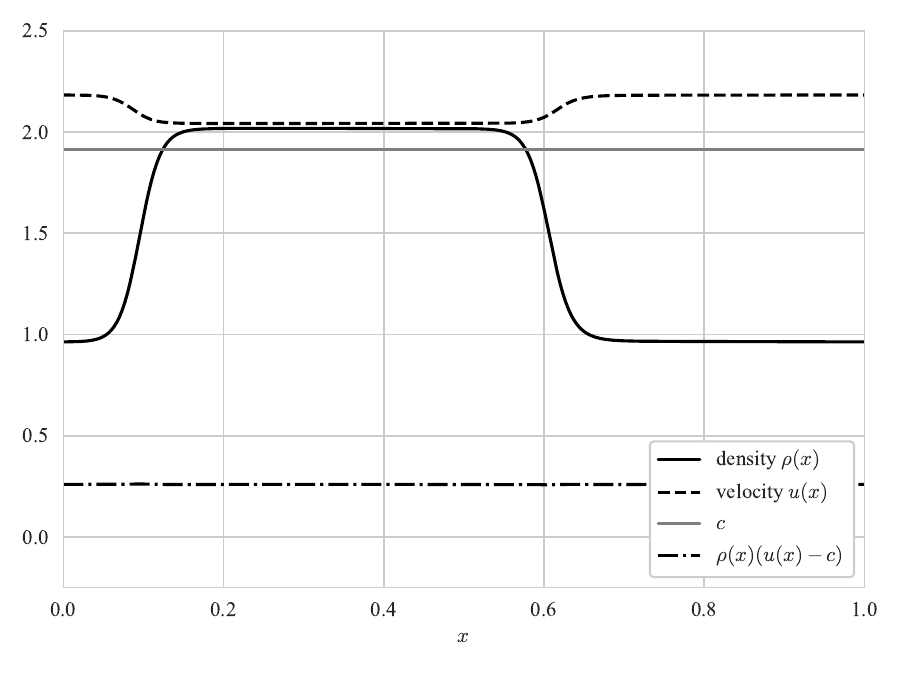}
  \caption{$ \rho(x, t) $, $ u(x, t) $, $ c(t) $, and $ \rho(x, t)(u(x, t) - c(t)) $ at $ t = 25 $. The parameters are set to $ \bar{\mu} = 10^{-3} $ and $ \varepsilon = 10^{-4} $. }
  \label{fig:mu3eps4}
\end{figure}

\begin{figure}[tbhp]
  \centering
  \includegraphics[width=0.5\textwidth]{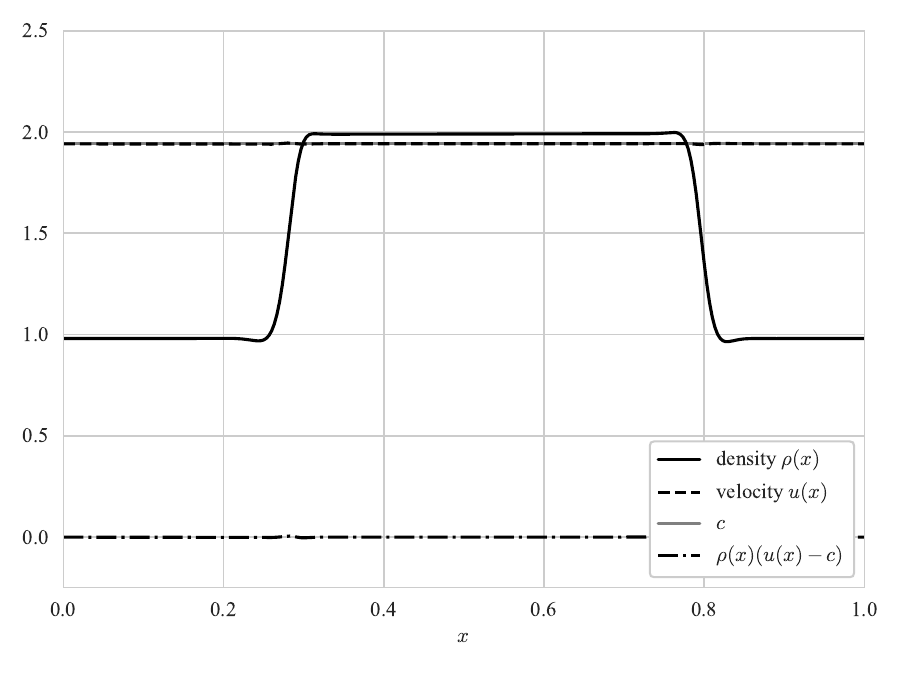}
  \caption{$ \rho(x, t) $, $ u(x, t) $, $ c(t) $, and $ \rho(x, t)(u(x, t) - c(t)) $ at $ t = 25 $. The parameters are set to $ \bar{\mu} = 10^{-1} $ and $ \varepsilon = 10^{-5} $. }
  \label{fig:mu1eps5}
\end{figure}

\begin{figure}[tbhp]
  \centering
  \includegraphics[width=0.5\textwidth]{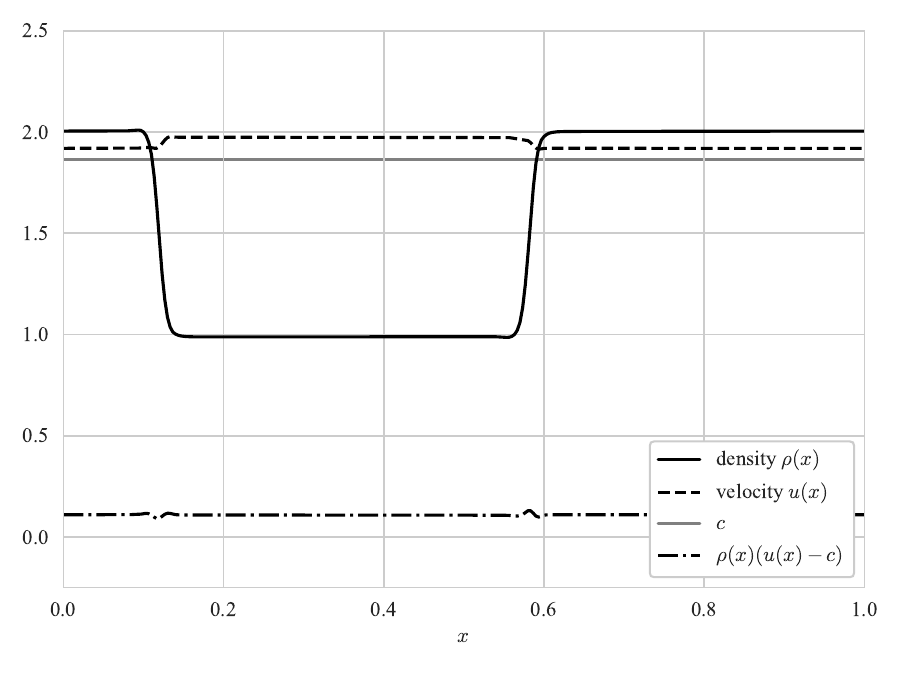}
  \caption{$ \rho(x, t) $, $ u(x, t) $, $ c(t) $, and $ \rho(x, t)(u(x, t) - c(t)) $ at $ t = 25 $. The parameters are set to $ \bar{\mu} = 10^{-3} $ and $ \varepsilon = 10^{-5} $. }
  \label{fig:mu3eps5}
\end{figure}
\begin{figure}[tbhp]
  \centering
  \begin{minipage}{0.4\columnwidth}
    \centering
    \includegraphics[width=\columnwidth]{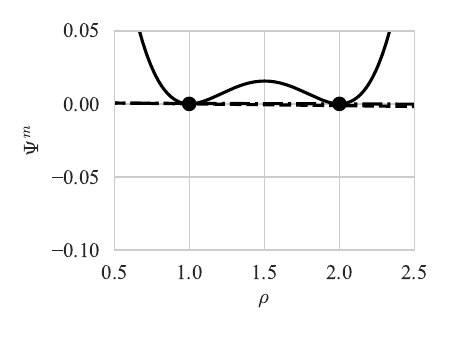}
    \vspace{-24pt}
    \subcaption{$ \bar{\mu} = 10^{-1} $,  $ \varepsilon = 10^{-4} $. }
    \label{fig:Psi_mu1eps4}
  \end{minipage}
  \begin{minipage}{0.4\columnwidth}
    \centering
    \includegraphics[width=\columnwidth]{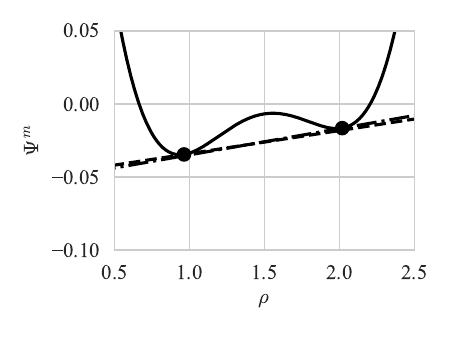}
    \vspace{-24pt}
    \subcaption{$ \bar{\mu} = 10^{-3} $,  $ \varepsilon = 10^{-4} $. } 
    \label{fig:Psi_mu3eps4} 
  \end{minipage}
  \centering
  \begin{minipage}{0.4\columnwidth}
    \centering
    \includegraphics[width=\columnwidth]{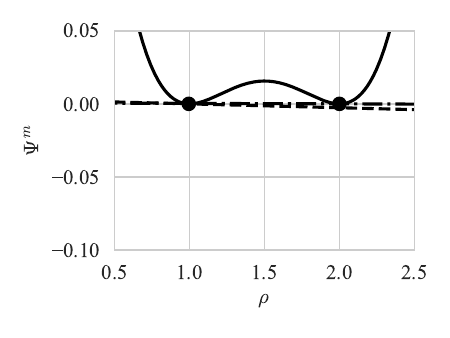}
    \vspace{-24pt}
    \subcaption{$ \bar{\mu} = 10^{-1} $,  $ \varepsilon = 10^{-5} $. }
    \label{fig:Psi_mu1eps5}
  \end{minipage}
  \begin{minipage}{0.4\columnwidth}
    \centering
    \includegraphics[width=\columnwidth]{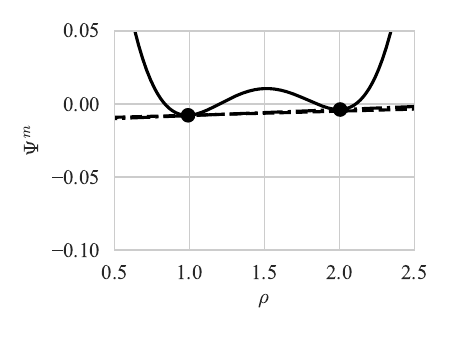}
    \vspace{-24pt}
    \subcaption{$ \bar{\mu} = 10^{-3} $,  $ \varepsilon = 10^{-5} $. }
    \label{fig:Psi_mu3eps5}
  \end{minipage}
  \caption{$ \Psi^m $ and its tangent line at $ \min_{x \in \mathbb{T}} \rho(x, t) $ and $ \max_{x \in \mathbb{T}} \rho(x, t) $ for $ t = 25 $. } 
  \label{fig:Psi} 
\end{figure} 

Lastly, we validated our method by testing it against an explicitly given exact solution. 
% Let us consider the initial-value problem \eqref{IVP1}--\eqref{IVP3} with $ \Psi(\rho) = \frac{1}{4} (\rho - 1)^2 (\rho - 2)^2 $, a positive constant $ \bar{\mu} $. 
Let $ k \in (0,1) $ be determined by $ \varepsilon = \frac{1}{64 (1 + k^2) (K(k))^2} $ and 
\begin{align}
  \tilde{\rho} (x) = \frac{3}{2} + \frac{1}{2} \left(\frac{2k^2}{1 + k^2}\right)^{1/2} \sn \left( \frac{4 K(k)}{\pi} x, k \right), 
\end{align}
where $ \sn $ represents the Jacobi elliptic function and $ K(k) = \int_{0}^{\pi / 2} \frac{d\theta}{\sqrt{1 - k^2 \sin^2 \theta}} $ is the complete elliptic integral of the first kind. For any constant $ \bar{u} $, $ (\rho(x,t), u(x,t)) = (\tilde{\rho}(x - \bar{u} t), \bar{u}) $ is an exact solution to the initial-value problem \eqref{IVP1}--\eqref{IVP3} with initial data $ (\rho_0, u_0) = (\tilde{\rho}, \bar{u}) $. 
We set $ \bar{u} = 2 $, and the numerical approximation is shown in Figure \ref{fig:exact}.
The numerical solution closely matches the exact solution, which confirms the accuracy of the numerical method. 

\begin{figure}[tbhp]
    \centering
    \includegraphics[width=0.5\textwidth]{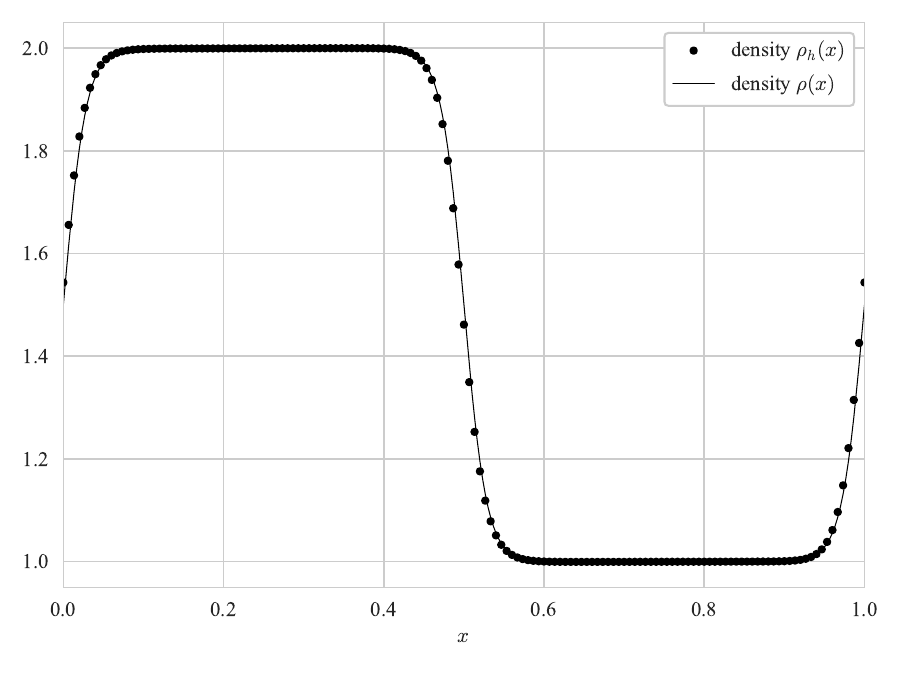}
    %\vspace{-30mm}
    \caption{The numerical solution $ \rho_h $ and the exact solution $ \tilde{\rho} $ at $ t = 1 $. Here we set $ \bar{\mu} = 0.1 $ and $ \varepsilon = 10^{-4} $. }
    \label{fig:exact}
\end{figure}

\section{Appendix: numerical scheme} \label{SApp}

In this section, we describe the numerical scheme used in Section \ref{SNE} for solving the one-dimensional isothermal Navier-Stokes-Korteweg equations. 
We decompose the original problem \eqref{IVP1} and \eqref{IVP2} into the following three subproblems: 

Subproblem 1: advection equation for density. 
\begin{equation}
    \partial_t \rho + \partial_x (u \rho) = 0, 
\end{equation}

Subproblem 2: advection equation for velocity. 
\begin{equation} 
    \partial_t u + u \partial_x u = 0, 
\end{equation} 

Subproblem 3: other terms in \eqref{IVP2}. 
\begin{equation}
    \partial_t u + \rho^{-1} \bar{\mu} \partial_x^2 u + \partial_\rho^2 \Psi \partial_x \rho - \varepsilon \partial_x^3 \rho = 0. 
\end{equation}

The subproblems are solved independently using suitable numerical approaches and are alternately treated according to the Strang splitting method, as described below. 

Let $ \rho_h^0 = I_h^3 \rho_0 $ and $ u_h^0 = I_h^3 u_0 $. Suppose that we have already obtained $ \rho_h^{n-1} $ and $ u_h^{n-1} $. Firstly, solve Subproblems 1 and 2 by the CIP method in $ [t^{n-1}, t^{n-1/2}] $. Let $ N_x \in \mathbb{N} $, $ h = 1 / N_x $ and $ x_j = jh $ for $ j = 0, \ldots, N_x $. 
We set the finite-dimensional function space $ V_h $ by 
\begin{align}
  V_h = \left\{v \in C^1(\mathbb{T}) \mathrel{}\middle|\mathrel{} v|_{[x_j,x_{j+1}]} \in \mathbb{P}_3([x_j,x_{j+1}]), \quad j = 0, \dots, M-1 \right\}. 
\end{align}
Here $\mathbb{P}_3$ denotes the space of cubic polynomials.
 Furthermore, let $ X[w](\cdot; \cdot, \cdot): [0,T] \times \mathbb{T} \times [0,T] \to \mathbb{R} $ represent the characteristic curve associated with a velocity $ w: \mathbb{T} \to \mathbb{R} $, that is, $ X(s; x, t) $ denotes the solution of the ordinary differential equation: 
\begin{empheq}[left={\empheqlbrace}]{alignat=2}
    & \frac{dX[w](s; x, t)}{ds} = w \left(X[w](s; x, t), s\right), & \quad & \text{in} \quad \mathbb{T} \times [t^{n-1}, t^{n-1/2}], \\
    & X[w](t; x, t) = x, & \quad & \text{in} \quad \mathbb{T}. 
\end{empheq}
The approximation of $ X[w](t^{n-1}; \cdot, t^{n-1/2}) $, denoted by $ X_h[w]^{n-1} $, is computed using the fourth-order Runge-Kutta method. The density and velocity are updated as follows: 
\begin{gather} 
    \rho_h^{n+1/2} = I_h^3 \left(\partial_x X_h[u_h^{n-1}] \cdot \rho_h^{n-1} \circ X_h[u_h^{n-1}]\right), \quad n = 1, \ldots, N, 
\end{gather} 
and
\begin{gather}
    u_h^{n+1/2,\ast} = I_h^3 \left(u_h^{n-1} \circ X_h[u_h^{n-1}]\right), \quad n = 1, \ldots, N,
\end{gather}
where $ I_h^3 $ denotes the cubic Hermite interpolation operator. 

We solve Subproblem 3 using an explicit finite difference scheme. 
Let $ \rho_j^{n,\ast} = \rho_h^{n,\ast}(x_j) $, $ \rho_{x,j}^{n,\ast} = \partial_x \rho_h^{n,\ast}(x_j) $, $ u_j^{n,\ast} = u_h^{n,\ast}(x_j) $ and $ u_{x,j}^{n,\ast} = \partial_x u_h^{n,\ast}(x_j) $ for $ j = 1, \ldots, N_x $. 
We compute the updates for the approximations of $ u $ and $ u_x $ as follows: 
\begin{gather}
  u_{j}^{n+1/2} = u_{j}^{n+1/2,\ast} + \Delta t \left(\frac{\bar{\nu} D_{h,j}^2 u^{n+1/2,\ast}}{\rho_j^{n+1/2}} - \frac{\rho_{x,j}^{n+1/2} \partial_\rho \Psi (\rho_j^{n+1/2})}{\rho_j^{n+1/2}} + \varepsilon D_{h,j}^3 \rho^{n+1/2} \right), \\
  \begin{aligned} 
    & u_{x,j}^{n+1/2} \\
		 & = u_{x,j}^{n+1/2,\ast} + \Delta t \left(\bar{\nu}\left(\frac{ D_{h,j}^3 u^{n+1/2,\ast}}{\rho_j^{n+1/2}} - \frac{\rho_{x,j}^{n+1/2} D_{h,j}^3 u^{n+1/2,\ast}}{(\rho_j^{n+1/2})^2}\right) - \frac{D_{h,j}^2 \rho^{n+1/2} \partial_\rho \Psi (\rho_j^{n+1/2})}{\rho_j^{n+1/2}} \right. \\
    & \quad \left. - \frac{(\rho_{x,j}^{n+1/2})^2 \partial_\rho^2 \Psi (\rho_j^{n+1/2})}{\rho_j^{n+1/2}} + \frac{(\rho_{x,j}^{n+1/2})^2\partial_\rho \Psi (\rho_j^{n+1/2})}{(\rho_j^{n+1/2})^2} + \varepsilon D_{h,j}^4 \rho^{n+1/2}\right), 
  \end{aligned}
\end{gather}
where the derivatives are approximated as 
\begin{gather}
  \begin{aligned}
     v_{xx} & \approx \frac{1}{h^2} \left[ \frac{128}{27} \left( \frac{v_{j+1} + v_{j-1}}{2} - v_j \right)+ \frac{7}{27} \left( \frac{v_{j+2} + v_{j-2}}{2} - v_j \right) \right. \\
    & \quad \left. - \frac{8h}{9} \left(v_{x,j+1} - v_{x,j-1} \right) - \frac{h}{36} \left( v_{x,j+2} - v_{x,j-2} \right) \right] \eqqcolon D_{h,j}^2 v, 
  \end{aligned} \\
  \begin{aligned}
    v_{xxx} & \approx \frac{1}{h^3} \left[\frac{176}{9} \left( \frac{v_{j+1} - v_{j-1}}{2} - h v_{x,j} \right) + \frac{31}{72} \left( \frac{v_{j+2} - v_{j-2}}{2} - 2h v_{x,j} \right) \right. \\
    & \quad \left. - \frac{16h}{3} \left( \frac{v_{x,j+1} + v_{x,j-1}}{2} - v_{x,j} \right) - \frac{h}{12} \left( \frac{v_{x,j+2} + v_{x,j-2}}{2} - v_{x,j} \right) \right] \\
		& \eqqcolon D_{h,j}^3 v, 
  \end{aligned}
\end{gather}
and
\begin{align}
  v_{xxxx} & \approx \frac{1}{h^4} \left[- \frac{128}{3} \left( \frac{v_{j+1} + v_{j-1}}{2} - h v_{j} \right) - \frac{41}{3} \left( \frac{v_{j+2} + v_{j-2}}{2} - 2h v_{j} \right) \right. \\
  & \quad \left. + 16 h \left( v_{x,j+1} - v_{x,j-1} \right) + \frac{3 h}{4} \left( v_{x,j+2} - v_{x,j-2} \right) \right] \eqqcolon D_{h,j}^4 v 
\end{align}
Here, $ v_j = v_h(x_j) $ and $ v_x = \partial_x v_h(x_j) $ for $ v = \rho_h^{n+1/2} $ or $ u_h^{n+1/2,\ast} $. Then we determine $ u_h^{n+1/2} \in V_h $ by $ u_h^{n+1/2}(x_j) = u_j^{n+1/2} $ and $ \partial_x u_h^{n+1/2}(x_j) = u_{x,j}^{n+1/2} $ for $ j = 1, \ldots, N_x $. 
Finally we solve Subproblems 1 and 2 in $ [t^{n-1/2}, t^n] $ again by the CIP method as 
\begin{gather} 
  \rho_h^{n} = I_h^3 (\partial_x X_h[u_h^{n-1/2}] \cdot \rho_h^{n-1} \circ X_h[u_h^{n-1/2}]), \\ 
  u_h^{n} = I_h^3 (u_h^{n-1} \circ X_h[u_h^{n-1/2}]). 
\end{gather}

% Lastly, we show that this numerical method simulates the exact solution. 

%%%%%%%%%%%%%%
\section*{Acknowledgements}
This work was done as a part of research activities of Social Cooperation Program ``Mathematical Science for Refrigerant Thermal Fluids'' sponsored by Daikin Industries, Ltd.\ at the University of Tokyo.
 The authors are grateful to members of the Technology and Innovation Center of Daikin Industries, Ltd.\ for showing several interesting phenomena related to phase transition with fruitful discussion which triggered this work.
 The work of the first author was partly supported by the Japan Society for the Promotion of Science (JSPS) through the grants KAKENHI: JP24H00183, JP24K00531 and by Arithmer Inc., Daikin Industries, Ltd.\ and Ebara Corporation through collaborative grants.
 The work of the second author was partly supported by JSPS through the grants KAKENHI: JP24K06860. 
 The work of the third author was supported by JSPS through the grants KAKENHI: JP24KJ0964.

\end{document}